\documentclass[12pt]{amsart}
\usepackage{amsmath,amscd,amssymb,amsfonts,latexsym}
\usepackage[all]{xy}
\usepackage{ulem}

\usepackage[latin1]{inputenc}

\usepackage[T1]{fontenc}
\usepackage{eucal} 
\usepackage{cmmib57} 
\usepackage{amsbsy}
\usepackage{amsthm,alltt}

\usepackage{graphicx,enumerate}

\usepackage{geometry}
\newcommand{\CP}{\mathbb{CP}^{n+1}}
\newcommand{\CN}{\mathbb{C}^{n+1}}
\newcommand{\U}{\mathcal{U}}
\newcommand{\C}{\mathbb{C}}
\newcommand{\Z}{\mathbb{Z}}
\newcommand{\F}{\mathbb{F}}
\newcommand{\Q}{\mathbb{Q}}
\newcommand{\K}{\mathcal{L}}
\newcommand{\R}{\mathcal{R}^{\bullet}}

\newtheorem{thm}{Theorem}[section]
\newtheorem{prop}[thm]{Proposition}
\newtheorem{lem}[thm]{Lemma}
\newtheorem{cor}[thm]{Corollary}
\theoremstyle{definition}
\newtheorem{definition}[thm]{Definition}
\newtheorem{example}[thm]{Example}
\theoremstyle{remark}
\newtheorem{remark}[thm]{Remark}
\numberwithin{equation}{section}

\newcommand{\MHM}{{\rm MHM}}

\def\be{\begin{equation}}
\def\ee{\end{equation}}

\def\bt{\begin{thm}}
\def\et{\end{thm}}

\def\bc{\begin{cor}}
\def\ec{\end{cor}}

\def\br{\begin{remark}}
\def\er{\end{remark}}

\def\bp{\begin{prop}}
\def\ep{\end{prop}}

\def\bl{\begin{lem}}
\def\el{\end{lem}}

\def\bex{\begin{example}}
\def\eex{\end{example}}

\def\bd{\begin{definition}}
\def\ed{\end{definition}}

\title{Reidemeister Torsion, Peripheral Complex, and Alexander Polynomials of Hypersurface Complements}
\author{Yongqiang Liu}
\address{Y. Liu: School of Mathematical Sciences, University of Science and Technology of China, No.96,  JinZhai Road Baohe District, Hefei, Anhui, 230026, P.R.China}
\email {liuyq@mail.ustc.edu.cn}

\author{Lauren\c{t}iu Maxim}
\address{L. Maxim: Department of Mathematics, University of Wisconsin, 480 Lincoln Drive, Madison, WI 53706, USA}
\email {maxim@math.wisc.edu}

\date{\today}


\subjclass[2010]{}

\begin{document}

\begin{abstract}  Let $f:\CN \rightarrow \C $ be a polynomial, which is transversal (or regular) at infinity. Let $\U=\CN\setminus f^{-1}(0)$ be the corresponding affine hypersurface complement. 
By using the peripheral complex associated to $f$,
we give several estimates for the (infinite cyclic) Alexander polynomials of $\U$ induced by $f$, and we describe the error terms for such estimates. The obtained polynomial identities can be further refined by using the Reidemeister torsion, generalizing a similar formula proved by Cogolludo and Florens in the case of plane curves.
We also show that the above-mentioned peripheral complex underlies an algebraic mixed Hodge module. This fact allows us to construct mixed Hodge structures on the Alexander modules of the boundary manifold of $\U$. \end{abstract}

\maketitle

\section{Introduction}
\subsection{Background}
Let $f:\CN \rightarrow \C $ be a polynomial map, and set $F_{0}=f^{-1}(0)$ and $\U=\CN \setminus F_{0}$. The topological study of the hypersurface $F_{0}$ and of its complement $\U$ is a classical subject going back to Zariski. Libgober introduced and studied Alexander-type invariants associated to the hypersurface complement $\U$, as induced by $f$. For $F_{0}$ a plane curve \cite{L1, L2}, or a hypersurface with only isolated singularities, including at infinity \cite{L3}, Libgober obtained a divisibility result, 
asserting that the only (possibly) non-trivial global  Alexander polynomial of $\U$ divides the product of the local Alexander polynomials associated with each singular point. 
 
More recently, the second named author used the intersection homology  theory in \cite{LM2} to provide generalizations of these results to the case of hypersurfaces with arbitrary singularities, provided that the defining equation $f$ is
transversal at infinity (i.e., the hyperplane at infinity is generic with respect to the projective completion of $F_0$). In particular, he proved a general divisibility result (cf. \cite[Theorem 4.2]{LM2}) relating the global and local Alexander polynomials. Furthermore, Dimca and Libgober \cite{DL} showed that for a polynomial transversal at infinity there exist canonical mixed Hodge structures on the (torsion) Alexander invariants of the hypersurface complement.  For more results related to Alexander-type invariants for complements of hypersurfaces with non-isolated singularities, see \cite{DM} and \cite{Liu}. 

A different approach to the study of Alexander polynomials relies on  the use of  Reidemeister torsion.
Milnor \cite{M2,M3} showed that the Alexander polynomial of a link coincides with the Reidemeister torsion of the link complement. This approach was further developed by Turaev \cite{Tu1} for the classical Alexander polynomial, and by Lin \cite{Lin} and Wada \cite{W} for twisted Alexander polynomials. 
Kirk and Livingston \cite{KL} extended this theory  to any finite CW-complex; in particular, they generalized Milnor's duality theorem for Reidemeister torsion. 

Cogolludo and Florens \cite{CF} studied twisted Alexander polynomials of plane algebraic curves by using the Reidemeister torsion, and obtained a polynomial identity involving global and local twisted Alexander polynomials. Specializing their result to the classical case (of the trivial representation), one obtains a geometric interpretation of Libgober's divisibility result.  

Let us assume from now on that $f$ is transversal at infinity. One of the goals of this paper is to provide a generalization to the case of hypersurface with non-isolated singularities of the Cogolludo-Florens identity for Alexander polynomials (cf. \cite[Corollary 5.8]{CF}). Our main tool will be the Cappell-Shaneson peripheral complex \cite{CS} associated to $f$. In more detail, we give a new description of the peripheral complex, from which we deduce several error estimates for the 
Alexander polynomials of the complement. Moreover, by exploiting the relation between the Alexander polynomials and Reidemeister torsion (\cite[Theorem 3.4]{KL}), we show how these estimates can be further refined by using the intersection form appearing in the duality for Reidemeister torsion.   

Our new description of the peripheral complex can also be used to show that the peripheral complex underlies an algebraic mixed Hodge module. In particular, after explaining the relation between the peripheral complex and the boundary manifold of the complement $\U$, we obtain mixed Hodge structures (MHS for short) on the Alexander modules of this boundary manifold. 

\subsection{Main results}
Unless otherwise specified, all homology and cohomology groups will be assumed to have $\Q$-coefficients.

Let $f:\CN\rightarrow \C$ be a degree $d$ polynomial. We say that $f$ is {\it transversal} (or {\it regular}) {\it at infinity} if $f$ is reduced, and the projective closure $V$ of $F_{0}$ in $\CP$ is transversal in the stratified sense to the hyperplane at infinity $H=\CP\setminus \CN$. 
Consider the infinite cyclic cover $\U^{c}$ of $\U$ corresponding to the kernel of the linking number homomorphism \begin{center}
 $f_{\ast}: \pi_{1}(\U)\rightarrow \pi_{1}(\C^{\ast})=\Z$
\end{center} induced by $f$. Then under the deck group action, each homology group $H_{i}(\U^{c})$ becomes a $\Gamma:=\Q[t,t^{-1}]$-module, called the {\it $i$-th Alexander module} of the hypersurface complement $\U$.  For $f$ transversal at infinity, the second named author showed that $H_{i}(\U^{c})$ is a torsion $\Gamma$-module for $i\leq n$ (cf. \cite[Theorem 3.6]{LM2}). We denote by $\delta_{i}(t)$ the corresponding (global) {\it Alexander polynomial}.

Let $N$ be an open regular neighborhood of $V\cup H$ in $\CP$ (cf. \cite{Du2}). Set $\U_{0}=\CP \setminus N$. Then $\U_{0}$ is homotopy equivalent to $\U$, and the boundary $\partial \U_{0}$ is a $(2n+1)$-dimensional real  manifold, called the {\it boundary manifold of $\U$}. The inclusion $\partial \U_{0} \hookrightarrow \U_{0}$ is an $n$-homotopy equivalence (cf. \cite[(5.2.31)]{D1}). Moreover, we have an epimorphism:  \begin{center}
$\rho: \pi_{1}(\partial \U_{0})\twoheadrightarrow \pi_{1}(\U_{0})= \pi_{1}(\U)\overset{f_*}{\twoheadrightarrow}  \pi_{1}(\C^{\ast})=\Z$,
\end{center}   which defines the infinite cyclic cover $(\partial \U_{0})^{c}$ of $\partial \U_{0}$. The related intersection form $\phi^{\rho}\in \Q(t)$ for the pair $(\U_0,\partial \U_0)$ is defined on $H_{n+1}( \U_{0}^{c})$, see (cf. \cite{KL}) or Section \ref{int} below. 

The {\it peripheral complex} $\R$ associated to $f$ (see \cite{CS,LM2} or Definition \ref{defpc} below) is a torsion $\Gamma$-module sheaf complex, which plays a key role in the second named author's generalizations of Libgober's results to the case of hypersurfaces with non-isolated singularities. Our first result is the following (see Proposition \ref{6.1}, Corollary \ref{3.3} and Corollary \ref{6.2}):

\begin{thm}  Assume that the polynomial $f: \CN\rightarrow \C$ is transversal at infinity. Then:
\begin{enumerate}
\item[(a)] there are $\Gamma$-module isomorphisms  \begin{center}
$H_{i}((\partial \U_{0})^{c})\cong H^{2n+1-i}(\CP;\R)$ for all $i$,
\end{center}
and, in particular,  $H_{i}((\partial \U_{0})^{c})$ is  a torsion $\Gamma$-module. Moreover, the zeros of the Alexander polynomial associated to $H_{i}((\partial \U_{0})^{c})$ are roots of unity for all $i$, and have order $d$ except for $i=n$. Finally, $H_{i}((\partial \U_{0})^{c})$ is a semi-simple $\Gamma$-module for $i\neq n$.  
\item[(b)] the peripheral complex $\R$ (when regarded as a complex of $\Q$-vector sheaves) is a (shifted) mixed Hodge module, hence the vector spaces $H_{i}((\partial \U_{0})^{c})$ inherit mixed Hodge structures from $\R$, for all $i$. Moreover, for $i\neq n$, the mixed Hodge structure on $H_{i}((\partial \U_{0})^{c})$ is compatible with the $\Gamma$-action, i.e., $t: H_{i}((\partial\U_{0})^{c})\rightarrow H_{i}((\partial\U_{0})^{c})$ is a mixed Hodge structure homomorphism.
\end{enumerate}
\end{thm}

Let $h=f_{d}$ be the top degree part of $f$, with corresponding Milnor fibre $F_{h}=\{h=1\}$, and denote by $h_{i}(t)$ the Alexander polynomial (or order) associated to the torsion $\Gamma$-module $H_{i}(F_{h})$. 
On the other hand, let $\psi_{f}\Q_{\CN}$ be the nearby cycle complex associated to $f$, and denote by  $\psi_{i}(t)$ the corresponding Alexander polynomial of $H^{2n-i}_{c}(F_{0},\psi_{f}\Q_{\CN})$.  In \cite[Theorem 1.1]{Liu}, the first named author studied the relation between the polynomials $\psi_{i}(t)$ and the Alexander polynomials $\delta_i(t)$ of the hypersurface complement $\U$. In particular, he showed that $\psi_i(t)=\delta_i(t)$ for $i<n$, and $\delta_n(t)$ divides $\psi_n(t)$. 
With the above notations, we have the following result, which establishes a more precise relationship between the polynomials $\delta_n(t)$ and $\psi_n(t)$ (see Theorem \ref{t6.2}):  

\begin{thm}\label{1.2} Assume that the degree $d$ polynomial $f: \CN\rightarrow \C$ is transversal at infinity. Let $\phi^{\rho}$ be the intersection form for $(\U_{0},\partial \U_{0})$ induced by $\rho$. Then, 
\begin{center}
$h_{n}(t)\cdot \psi_{n}(t)= \delta_{n}^{2}(t)\cdot \det(\phi^{\rho})$.
\end{center}  Moreover, we have the following degree estimates \footnote{Recall that the total degree of a Laurent polynomial in $\Q[t,t^{-1}]$ is defined as the difference between the highest and resp. the lowest power of $t$ (with non-zero coefficients). In particular, unit elements $ct^k$ ($c \in \Q$, $k \in \Z$) of $\Q[t,t^{-1}]$ have total degree zero. The total degree of a product of Laurent polynomials is the sum of the total degrees of the factors. The degree of an element $P/Q \in \Q(t)$ (with $P, Q \in \Q[t,t^{-1}]$) is defined as the difference between the total degrees of $P$ and $Q$ respectively.}:
$$\deg( \det(\phi^{\rho}))\leq 2 d \cdot\mu,$$ where $\mu=\vert \chi  (\U) \vert$.
\end{thm}

As an application to the case of polynomials with only isolated singularities, we obtain  the following generalization of Corollary 5.8 from \cite{CF}, and a new obstruction on the (degree of the) intersection form:

\begin{cor}\label{1.3} Let $f: \CN\rightarrow \C$ be a degree $d$ polynomial, which is transversal at infinity.  Assume that the hypersurface $F_0=\{f=0\}$ has only isolated singularities. Then we have the following polynomial identity:
\begin{equation}
(t-1)^{(-1)^{n+1}(1+\chi(\U))}(t^{d}-1)^{\xi}\cdot \prod_{p\in Sing(F_{0})}\Delta_{p}(t)= \delta_{n}(t)^{2} \cdot det(\phi^{\rho})
\end{equation}
 where $ \Delta_{p}(t)$ is the top local Alexander polynomial associated to the point $p\in Sing(F_{0})$, and $\xi=\dfrac{(d-1)^{n+1}+(-1)^{n}}{d}$.  Moreover, the degree of the polynomial $\det(\phi^{\rho})$ is even.
\end{cor}

\subsection{Summary}

The paper is organized as follows. 

In Section \ref{prelim}, we recall the definition and main results on the Alexander modules, peripheral complex and the Sabbah specialization complex.
In section \ref{per}, we give a new description of the peripheral complex associated with a hypersurface. As a byproduct, we show that the peripheral complex underlies a (shifted) algebraic mixed Hodge module. In section \ref{error}, we give several estimates for the Alexander polynomials of the hypersurface complement and study  the error terms for such estimates. 
Section \ref{Rtors} recalls the basic constructions and main results on the Reidemeister torsion of a finite CW-complex, and in particular, the duality theorem and the intersection form for the torsion. In Section \ref{boundary}, we introduce the boundary manifold $\partial \U_0$ of the hypersurface complement, and we describe its (linking number) Alexander modules $H_{i}((\partial\U_{0})^{c})$ in terms of the peripheral complex. In particular, we endow these Alexander modules $H_{i}((\partial\U_{0})^{c})$ with mixed Hodge structures. 
Finally, Section \ref{Alexid} is devoted to the proof of both Theorem \ref{1.2} and Corollary \ref{1.3}. 

\bigskip

\textbf{Acknowledgments.} We are grateful to Alex Dimca and J\"org Sch\"urmann for useful discussions.
The first named author is supported by China Scholarship Council (file No. 201206340046). He thanks the Mathematics Department at the University of Wisconsin-Madison for hospitality during the preparation of this work. The second named author is partially supported by grants from NSF (DMS-1304999), NSA (H98230-14-1-0130),  Simons Foundation (\#277891), and by a grant of the Ministry of National
Education, CNCS-UEFISCDI project number PN-II-ID-PCE-2012-4-0156.

%%%%%%%%%%%%%%%%%%%%%%%%%%%%%%%%%%%%
%%%%%%%%%%%%%%%%%%%%%%%%%%%%%%%%%%%%

\section{Preliminaries}\label{prelim}

\subsection{Alexander modules}

Let $f=f(x_{1},\cdots,x_{n+1}):\CN \rightarrow \C  $ be a reduced degree $d$ polynomial map, and set $F_{0}=f^{-1}(0)$ and $\U=\CN\setminus F_{0} $. 
We say that $f$ is {\it transversal at infinity} if the projective closure $V$ of $F_{0}$ in $\CP$ is transversal in the stratified sense to the hyperplane at infinity $H=\CP\setminus \CN=\lbrace x_{0}=0 \rbrace$. If $f$ is transversal at infinity, the affine hypersurface $ F_{0} $ is homotopy equivalent to a bouquet of $n$-spheres, i.e.,  
\begin{equation}\label{join} F_{0}\simeq \vee_{\mu} S^{n},\end{equation} where $\mu$ denotes the number of spheres in the above join (cf. \cite[page 476]{DP}). It is shown in loc.cit. that $\mu$ can be determined topologically as the degree of the gradient map associated to the homogenization $\tilde{f}$ of $f$.

We have a surjective homomorphism: $\pi_{1}(\U)\rightarrow \pi_{1}(\C^{\ast})=\Z  $ induced by $f$, which shall be called the {\it linking number homomorphism} (see \cite[page 76-77]{D1} for a justification of terminology).
Let us consider the corresponding infinite cyclic cover $\U^{c}$ of $\U$. Then under the deck group action, every homology group $H_{i}(\U^{c},\Q)$ becomes a $\Gamma:=\Q[t,t^{-1}]$-module.
\begin{definition}  The $\Gamma$-module $ H_{i}({\U}^{c})$  is called the {\it $i$-th Alexander module} of the hypersurface complement $\U$. 
\end{definition}
When $H_{i}(\U^{c})$ is a torsion $\Gamma$-module, we denote by $\delta_{i}(t)$ the corresponding {\it Alexander polynomial} (also called {\it order} in \cite{M4}). Since $\U$ has the homotopy type of a finite $(n+1)$-dimensional CW complex, $H_{i}(\U^{c})=0$ for $i>n+1$ and $H_{n+1}(\U^{c})$ is a free $\Gamma$-module. Hence the only interesting Alexander modules $H_{i}(\U^{c})$ appear in the range $0\leq i\leq n$, and the following result holds:
   
 \begin{thm}(\cite[Theorems 3.6, 4.1]{LM2} \label{t2.2}  Assume that $f: \CN\rightarrow \C$ is a reduced, degree $d$ polynomial, which is transversal at infinity.  Then $H_{i}({\U}^{c}) $ is a finitely generated semi-simple torsion $\Gamma$-module for $ 0 \leq i\leq n$, and the roots of the corresponding Alexander polynomial $\delta_{i}(t)$ are roots of unity of order $d$.
 \end{thm}
      
\begin{remark} 
The second named author showed in \cite{LM2} that $H_{0}(\U^{c})\cong \Gamma/(t-1)$, and $H_{n+1}({\U}^{c}) $ is a free $\Gamma$-module of rank $\vert\chi(\U)\vert$. On the other hand, by using the additivity of the Euler characteristic, it is easy to see from (\ref{join}) that $\chi(\U)=(-1)^{n+1}\mu$. Therefore,  \begin{equation}\label{free} H_{n+1}({\U}^{c})\cong \Gamma^{\mu}.\end{equation}
\end{remark}

%%%%%%%%%%%%%%%%%%%%%%%%%%%%%%%%%%

 \subsection{Linking number local system} \label{ls} 
 Let us consider the local system $\K $ on $\U$ with stalk $\Gamma$, and representation of the fundamental group defined by the composition: 
$$ \pi_{1}(\U) \overset{f_*}{\rightarrow} \pi_{1}(\C^{\ast})   \rightarrow   Aut(\Gamma),$$
with the second map being given by  $1_{\Z}\mapsto t$. 
Here $t$ is the automorphism of $\Gamma$ given by multiplication by $t$. $\K$ shall be referred to as the {\it linking number local system}.

If $A^{\bullet}$ is a complex of $\Gamma$-sheaves, let $\mathcal{D} A^{\bullet}$ denote its Verdier dual. Then we have that:
$$\mathcal{D} \K  \cong \K^{op}[2n+2],$$
 where $\K^{op}$ is the local system obtained from $\K$ by composing all $\Gamma$-module structures with the involution $t \mapsto t^{-1}$.
  
In terms of the local system $\K$, we have the following $\Gamma$-module isomorphisms (\cite[Corollary 3.4]{LM2}): 
\begin{equation}\label{max1} H^{2n+2-i}_{c}(\U,\K)\cong H_{i}(\U,\K)\cong H_{i}(\U^{c})\end{equation} 
for all $i$.
Similarly, \be\label{max2} H_{i}(\U,\K^{op})\cong \overline{ H_{i}(\U^{c})},\ee where $\overline{\ast}$ denotes the composition with the involution $t \rightarrow t^{-1}$. By using the Universal Coefficient Theorem (e.g., see \cite[Theorem 3.4.4]{B}), we also obtain:  
 \begin{equation}\label{uct}
 H^{i+1}(\U;\K)\cong {\rm Free} \left( H_{i+1}(\U, \K^{op}) \right)\oplus {\rm Torsion} \left( H_{i}(\U,\K^{op})\right).
 \end{equation}

%%%%%%%%%%%%%%%%%%%%%%%%%%%%%%%%%%

\subsection{The peripheral complex}\label{pc}

For any complex algebraic variety $X$ and any ring $R$, we denote by  $D^{b}_{c}(X,R)$ the derived category of bounded cohomologically $R$-constructible complexes of sheaves on $X$. For a quick introduction to derived categories, the reader is advised to consult \cite{D2}.
 
By choosing a Whitney stratification of $V$, and using the transversal hyperplane at infinity $H$, we obtain a stratification of the pair $(\CP, V \cup H)$.  Then, for any perversity function $\overline{p}$, the intersection homology complex $IC_{\overline{p}}^{\bullet}(\CP,\K) \in D^{b}_{c}(\CP,\Gamma)$ is defined by using Deligne's axiomatic construction (see \cite{B, GM2}). In this paper, we mainly use the indexing conventions from \cite{GM2}. In particular, we have the following normalization property: $IC_{\overline{p}}^{\bullet}(\CP,\K)|_{\U}\cong \K[2n+2]$.

Let us recall the following result:
\begin{thm}\label{thmax} (\cite[Lemma 3.1]{LM2}) Assume that the polynomial $f: \CN\rightarrow \C$ is transversal at infinity. Let $j$ be the inclusion of $\U$ in $\CP$. Then we have the following quasi-isomorphisms in $D^{b}_{c}(\CP,\Gamma)$:
 \be IC_{\overline{m}}^{\bullet}(\CP,\K)\cong j_{!}\K[2n+2], \ee
\be IC_{\overline{l}}^{\bullet}(\CP,\K)\cong Rj_{\ast}\K[2n+2],\ee
 where the middle and logarithmic perversities are defined as: $\overline{m}(s)=[(s-1)/2]$ and resp. $\overline{l}(s)=[(s+1)/2]$. (Note that $\overline{m}(s)+\overline{l}(s)=s-1$, i.e., $\overline{m}$ and $\overline{l}$ are superdual perversities, in the sense of \cite{CS}.)
  \end{thm}
In the above notations, the Cappell-Shaneson superduality isomorphism
can be stated as (cf. \cite[Theorem 3.3]{CS}):
\be IC_{\overline{m}}^{\bullet}(\CP,\K)^{op}\cong \mathcal{D} (IC_{\overline{l}}^{\bullet}(\CP,\K))[2n+2] ,\ee where if $A$ is a complex of sheaves, $A^{op}$ is the $\Gamma$-module obtained from the $\Gamma$-module $A$ by composing all module structures with the involution $t \to t^{-1}$.

\begin{definition}\label{defpc}
 The {\it peripheral complex} $\R \in D_{c}^{b}(\CP, \Gamma) $ is defined by the distinguished triangle (see \cite{CS}) \begin{equation}
 IC_{\overline{m}}^{\bullet}(\CP,\K)\rightarrow IC_{\overline{l}}^{\bullet}(\CP,\K)\rightarrow \R[2n+2] \overset{[1]}{\rightarrow},
  \end{equation} or, by using Theorem \ref{thmax}, by 
  \begin{equation}
  j_{!}\K \rightarrow Rj_{\ast}\K \rightarrow \R  \overset{[1]}{\rightarrow}.
  \end{equation}
  \end{definition}
Then, up to a shift,  $\R$ is a self dual (i.e., $ \R \cong \mathcal{D}\mathcal{R}^{\bullet op}[-2n-1] $), torsion (i.e., the stalks of its cohomology sheaves are torsion modules) perverse sheaf on $\CP$ (see \cite[section 3.2]{LM2}). In fact, $\R$ has compact support on $V\cup H$, and 
\be\label{sup} \R\vert _{V\cup H} \cong  (Rj_{\ast} \K)\vert_{V\cup H}.\ee 
\begin{remark} The peripheral complex $\R$ as defined here corresponds to  $\R[-2n-2]$  in the notations of Cappell and Shaneson, see \cite{CS} or \cite{LM2}. 
\end{remark}

%%%%%%%%%%%%%%%%%%%%%%%%%%%%%%%%%

 \subsection{The Sabbah specialization complex}\label{sp} The Sabbah specialization complex (cf. \cite{Sab}, and its reformulation in \cite{Bu2}) can be regarded as a generalization of Deligne's nearby cycle complex. For a quick introduction to the theory of nearby cycles, the reader is advised to consult \cite{D2} and \cite{Ma}. 
 
Let us recall the relevant definitions. Consider the following commutative diagram of spaces and maps:\begin{center}
$\xymatrix{
F_{0} \ar[r]^{i}  \ar[d]^{f}  & \CN  \ar[d]^{f}   & \ar[l]_{l} \U \ar[d]^{f}  & \ar[l]_{\pi} \U^{c} \ar[d]^{\widehat{f}} \\
        \{0\}    \ar[r]      & \C       & \ar[l] \C^{\ast}      & \ar[l]_{\widehat{\pi}} \widehat{\C^{\ast}} 
}$
\end{center}
where $\widehat{\pi}$ is the universal covering of the punctured disk $\C^{\ast}$, and the right-hand square of the diagram is cartesian.

\begin{definition} The {\it  Sabbah specialization functor}  of $f$ is defined by 
$$\psi^{S}_{f} =i^{\ast}Rl_{\ast} R\pi_{!} (l \circ \pi)^{\ast} : D^{b}_{c}(\CN,\Q) \rightarrow D^{b}_{c}(F_{0},\Gamma),$$ 
 and we call $\psi^{S}_{f} \Q_{\CN}$ the {\it Sabbah specialization complex}.
\end{definition}
\br The definition of the Sabbah specialization complex  is slightly different from that of the nearby cycle complex, where $R\pi_{!}$ is replaced by $R\pi_{\ast}$. \er

For short, in the following we write $\Q$ for the constant sheaf $\Q_{\CN}$ on $\CN$.

Consider the natural forgetful functor
$${\rm\it for}:  D^{b}_{c}(F_{0},\Gamma)\rightarrow D^{b}_{c}(F_{0},\C),$$
which maps a torsion $\Gamma$-module sheaf complex to its underlying $\Q$-complex. Let $\psi_{f}\Q$ be the Deligne nearby cycle complex associated to $f$. It is known that  one has a non-canonical isomorphism (cf. \cite[page 13]{Br}): 
\be\label{for} {\rm\it for} \circ \psi^{S}_{f} (\Q) \cong  \psi_{f}\Q [-1].\ee

The next result is a direct consequence of \cite[ Lemma 3.4(b)]{Bu2}.
\begin{lem}(\cite[Section 2.4]{Liu}) We have a quasi-isomorphism in $D^{b}_{c}(F_{0},\Gamma)$:
\be \R\vert_{F_{0}}\cong \psi^{S}_{f}{\Q}.\ee 
Moreover, we have the following distinguished triangle in $D^{b}_{c}(\CN,\Gamma)$: \begin{equation}\label{2.4}
l_{!}\K \rightarrow Rl_{\ast}\K \rightarrow   i_{!}\psi^{S}_{f}\Q \overset{[1]}{\rightarrow}.
\end{equation}
\end{lem}

%%%%%%%%%%%%%%%%%%%%%%%%%%%%%%%%%%%%%
%%%%%%%%%%%%%%%%%%%%%%%%%%%%%%%%%%%%%  

\section{Peripheral Complex as a Mixed Hodge Module}\label{per}
In this section, we give a new characterization of the peripheral complex, and show that (up to a shift) it underlies a mixed Hodge module. For a quick introduction to the category of mixed Hodge module, the reader is advised to consult \cite{Sa}.

Let $h=f_{d}$ be the top degree part of $f$, with corresponding Milnor fibre $F_{h}=\{h=1\}$. Then, it is shown in \cite{LM2} that $\R\vert_{H \setminus H\cap V}$ is a local system  $\K(h)$ with stalk $\Gamma/(t^{d}-1)$ placed in degree $1$,  i.e., 
 \be \R\vert_{H\setminus H\cap V}\cong \K(h)[-1].\ee

\begin{thm}\label{new}  Assume that the polynomial $f: \CN\rightarrow \C$ is transversal at infinity, and let $V$ be the projective completion of $F_0=\{f=0\}$.  Let $i_{v}$ be the inclusion of $F_{0}$ into $V$, and $i_{h}$ be the inclusion of $H \setminus V\cap H$ into $H$. Then
\begin{equation}\label{3.1}
\R\vert_{V}\cong Ri_{v \ast} \psi^{S}_{f}\C.
\end{equation}
\begin{equation}\label{3.2}
\R\vert_{H} \cong Ri_{h \ast}\K(h)[-1].
\end{equation}
\end{thm}
\begin{proof} Let us only prove (\ref{3.1}), as (\ref{3.2}) can be obtained in a similar manner.
 Consider the following commutative diagram of inclusions:
\begin{center}
$\xymatrix{
\U  \ar[r]^{l}    \ar[d]^{k^{\prime}}  \ar[rd]^{j}          & \CN \ar[d]^{k}  \\
\CP \setminus V         \ar[r]^{l^{\prime}}       & \CP
}$
\end{center}
Since $V$ intersects $H$ transversally, there exists a base change isomorphism associated with this diagram (\cite[Lemma 6.0.5]{Sc}):
\begin{equation}\label{3.3b}
l^{\prime}_{!}\circ Rk^{\prime}_{\ast}=Rk_{\ast}\circ l_{!}.
\end{equation}
Let $v$ be the inclusion of $V$ into $\CP$.
Note that $Rl^{\prime}_{\ast} Rk^{\prime}_{\ast}\K=Rj_{\ast}\K$ and, by Section \ref{pc}, we have that $\R\vert_{V}=(Rj_{\ast}\K)\vert_{V}$. Then we have a distinguished triangle:\begin{equation}\label{3.4}
l^{\prime}_{!}Rk^{\prime}_{\ast}\K \rightarrow Rl^{\prime}_{\ast} k^{\prime}_{\ast}\K \rightarrow v_{!}(\R\vert_{V}) \overset{[1]}{\rightarrow}.
\end{equation}
Using the base change isomorphism (\ref{3.3b}) and the commutativity of the above diagram, the distinguished triangle (\ref{3.4}) can be written as:
 \begin{equation}\label{3.5}
Rk_{\ast} l_{!}\K \rightarrow Rk_{\ast} Rl_{\ast}\K \rightarrow v_{!}(\R\vert_{V}) \overset{[1]}{\rightarrow}.
\end{equation} 
Recall now that there is a distinguished triangle (\ref{2.4}):
\be l_{!}\K \rightarrow Rl_{\ast}\K \rightarrow i_{!} \psi^{S}_{f}\Q,\ee
 where $i$ is the inclusion of $F_{0}$ into $\CN$. By applying the functor $Rk_{\ast}$ to this triangle, we obtain the distinguished triangle:
 \begin{equation} \label{3.6}
Rk_{\ast} l_{!}\K \rightarrow Rk_{\ast} Rl_{\ast}\K \rightarrow Rk_{\ast}i_{!} \psi^{S}_{f}\Q \overset{[1]}{\rightarrow} .
\end{equation}
So, by comparing the two triangles (\ref{3.5}) and (\ref{3.6}), we get the following quasi-isomorphism:  
\begin{equation}\label{3.7}
v_{!}(\R\vert_{V})\cong Rk_{\ast} i_{!} \psi^{S} _{f}\Q.
\end{equation}
Since $F_{0}$ is closed in $\CN$, we have $i_{!}=Ri_{\ast}$.   So for $i_{v}$ the inclusion of $F_{0}$ into $V$, we have $v\circ i_{v}=k\circ i$, hence
\be\label{3.8} Rk_{\ast}i_{!} \psi^{S}_{f}\Q=Rk_{\ast}Ri_{\ast} \psi^{S}_{f}\Q=Rv_{\ast}Ri_{v \ast} \psi^{S}_{f}\Q.\ee
Finally, by applying $v^{\ast}$ to equation (\ref{3.7}), and using (\ref{3.8}) and the standard identities  $v^{\ast}v_{!}=id$ and $v^{\ast}Rv_{\ast}=id$, we get: 
\be \R\vert_{V}\cong Ri_{v \ast}\psi^{S}_{f}\Q.\ee
\end{proof}

The result of Theorem \ref{new} above can be used to endow the peripheral complex with  a mixed Hodge module structure.
Recall that  there is a  natural forgetful functor  ${\rm\it for}:  D^{b}_{c}(X,\Gamma)\rightarrow D^{b}_{c}(X,\Q)$, which assigns to a torsion complex of $\Gamma$-sheaves its underlying $\Q$-complex. In what follows, we will use the same notation for a $\Gamma$-complex $\mathcal{A}^{\bullet}$ and for its underlying $\Q$-complex ${\rm\it for}(\mathcal{A}^{\bullet})$.

After applying the forgetful functor to (\ref{3.1}), and by using (\ref{for}), we obtain the following quasi-isomorphism of complexes of $\Q$-sheaves:
\be\label{ds} \R \vert_{V}\cong Ri_{v \ast}\psi_{f}\Q[-1] .\ee
Also note that the local system $\K(h) \in D^{b}_{c}(H\setminus H\cap V ,\Q)$ is induced by the natural $d$-fold cover map $p: F_{h}\rightarrow (H\setminus V\cap H)$ or, more precisely, 
\be\label{loc}
\K(h)\cong  Rp_{\ast}\Q_{F_{h}}\in D^{b}_{c}(H \setminus H\cap V ,\Q).\ee

\begin{remark} Since $h=f_d$ is a degree $d$ homogeneous polynomial function on $\CN$, the hypersurface $V_{h}=\overline{\lbrace h=0 \rbrace}\subseteq \CP$ is already transversal to the hyperplane at infinity $ H=\lbrace x_{0}=0 \rbrace $. Let $\R_{h}$ be the peripheral complex  associated to $h$. 
Then  (\ref{3.2}) implies that $\R\vert_{ H}=\R_{h }\vert _{H}$.
 \end{remark}

We can now prove the following result:
 \begin{cor} \label{3.3} The peripheral complex $\R$ underlies a (shifted) algebraic mixed Hodge module.
 \end{cor}
 \begin{proof}
For the purpose of this proof only, we switch to the perverse conventions used in Saito's theory, according to which $\R[n+1]$ is a perverse sheaf on $\CP$. All sheaf complexes appearing  in this proof will be complexes of $\Q$-sheaves (i.e., we apply the forgetful functor to all $\Gamma$-sheaf complexes).

Consider the inclusions $H \setminus V \cap H \overset{s}{\hookrightarrow} V \cup H \overset{r}{\hookleftarrow} V$,
and the associated distinguished triangle in $D^b_c(V \cup H,\Q)$:
\be s_!\R|_{H \setminus V \cap H}[n+1] \to \R[n+1] \to r_*\R|_V[n+1] \overset{[1]}{\to}
\ee
Recall that $\R|_{H \setminus V \cap H} \cong \K(h)[-1]$, while by (\ref{ds}) we have that: $\R|_V \cong Ri_{v \ast}\psi_{f}\Q[-1]$. Since $\R|_{H \setminus V \cap H}[n+1] \cong \K(h)[n]$ is a perverse sheaf on $H \setminus V \cap H$, and $s$ is a quasi-finite affine map, it follows from \cite[Corollary 5.2.17]{D2} that $s_!\R|_{H \setminus V \cap H}[n+1]$ is a perverse sheaf on $V \cup H$. Moreover, we deduce by (\ref{loc}) that $\K(h)[n]$, hence also $s_!\R|_{H \setminus V \cap H}[n+1]$ underly algebraic mixed Hodge modules. Since $\Q[n+1]$ is a perverse sheaf on $\C^{n+1}$ underlying a mixed Hodge module, and the functor $\psi_f[-1]$ preserves perverse sheaves (and mixed Hodge modules), it follows that $(\psi_{f}[-1])(\Q[n+1])$ is a perverse sheaf on $F_0$ underlying  a mixed Hodge module. Moreover, as $i_v$ is a quasi-finite affine morphism, it follows as above that $\R|_V[n+1]\cong Ri_{v \ast}(\psi_{f}[-1])(\Q[n+1])$ is a perverse sheaf on $V$ underlying  a mixed Hodge module. Finally, since $r$ is proper, $r_!=r_*$ preserves perverse sheaves and resp. mixed Hodge modules, so $r_*\R|_V[n+1]$ is a perverse sheaf on $V \cup H$ underlying a mixed Hodge module.

The above considerations show that $\R[n+1]$ can be regarded as an extension of perverse sheaves, both of which underly mixed Hodge modules. So $\R[n+1]$ is 
 an element in the first Yoneda Extension group $Y{\rm Ext}^1({\rm For}({C}),{\rm For}(A))$, for suitable mixed Hodge modules $C$ and $A$ as described above, and where ${\rm For}:\MHM \to {\rm Perv}_{\Q}$ denotes the forgetful functor assigning to a mixed Hodge module the corresponding rational sheaf complex. Since Yoneda Ext groups $Y{\rm Ext}^i$ agree with the derived category ${\rm Ext}$ groups
${\rm Ext}^i={\rm Hom}(-,-[i])$ for noetherian or artinian abelian categories
such as  $\MHM$ or ${\rm Perv}_{\Q}$, and the forgetful functor $${\rm For}: {\rm Ext}^i(C,A) \to {\rm Ext}^i({\rm For}({C}), {\rm For}(A))$$ is surjective for all $i$ for
given mixed Hodge modules $A$ and $C$ (cf. \cite[Theorem 2.10]{Sa1}), it follows  that $\R[n+1]$ underlies a mixed Hodge module.
 \end{proof}

%%%%%%%%%%%%%%%%%%%%%%%%%%%%%%%%%%%%%
%%%%%%%%%%%%%%%%%%%%%%%%%%%%%%%%%%%%%

\section{Error estimates for Alexander polynomials}\label{error}
In this section, we give several error estimates for Alexander polynomials of hypersurface complements.

\bp  \label{p4.1} In our notations, we have  $\Gamma$-module isomorphisms  \begin{equation}\label{4.1}
H^{2n+1-i}(\CP;\R) \cong \left\{ \begin{array}{ll}
H_{i}(\U^{c}), & i<n, \\
\overline{ H_{2n-i}(\U^{c})}, & i>n,\\
\end{array}\right.
\end{equation}
 and an exact sequence of  $\Gamma$-modules for $i=n$: \begin{equation}\label{4.2}
0\rightarrow \Gamma^{\mu}\rightarrow \Gamma^{\mu}\oplus \overline{ H_{n}({\U}^{c})} \rightarrow  H^{n+1}(\CP;\R)\rightarrow H_{n}({\U}^{c}) \rightarrow 0
  \end{equation} 
\ep

\begin{proof}
Consider the distinguished triangle $$Rj_{!}\K\rightarrow Rj_{\ast}\K\rightarrow \R \overset{[1]}{\rightarrow}$$ of Definition \ref{defpc}.
By  applying the hypercohomology with compact support functor, we have the following long exact sequence:
$$\cdots \rightarrow  H^{2n+1-i}(\CP;\R) \rightarrow H^{2n+2-i}(\CP;Rj_{!}\K) \rightarrow H^{2n+2-i}(\CP;Rj_{\ast}\K) \rightarrow  %H^{2n+2-i}(\CP;\R) \rightarrow 
\cdots $$
The claim follows from the above sequence together with the following  $\Gamma$-isomorphisms from Section \ref{ls}:
\begin{enumerate}
\item[(a)]  $H^{2n+2-i}(\CP;Rj_{!}\K)\cong H_c^{2n+2-i}(\U;\K) \cong H_{i}(\U^{c})$,
\item[(b)] $ H^{2n+2-i}(\CP; Rj_{\ast}\K)\cong H^{2n+2-i}(\U;\K) \cong \left\{ \begin{array}{ll}
0, & i<n+1 ,\\
\Gamma^{\mu}\bigoplus \overline{ H_{n}(\U^{c})}, & i=n+1,\\
\overline{ H_{2n+1-i}(\U^{c})}, & i>n+1.\\
\end{array}\right.$
\end{enumerate}
\end{proof}

Recall that $\delta_i(t)$ denotes the Alexander polynomial associated to the Alexander module $H_{i}(\U^{c})$ ($i \leq n$). 
Let $r_{i}(t)$ be the Alexander polynomial associated to the torsion $\Gamma$-module $H^{2n+1-i}(\CP;\R)$. The above proposition yields the following relationship between the polynomials $r_i$ and $\delta_i$:
\bc \label{c4.2}
 \begin{equation}\label{4.3}
 r_{i}(t)=\left\{ \begin{array}{ll}
\delta_{i}(t), & i<n, \\
\overline{ \delta_{2n-i}(t)}, & i>n.\\
\end{array}\right.
 \end{equation}
 and $\delta_{n}(t) \cdot\overline{\delta_{n}(t)}$ divides $r_{n}(t)$. 
 \ec
Set $$\varphi(t)= \dfrac{r_{n}(t)}{\delta_{n}(t)\cdot \overline{\delta_{n}(t)}}.$$

Let $F_h$ denote as before the Milnor fiber associated to the polynomial $h=f_d$, the top degree part of the polynomial $f$. 
Let $h_{i}(t)$ be the Alexander polynomial associated to $H_{i}(F_{h})$. Then it was shown in \cite[Theorem 4.7]{LM2} that  $h_{i}(t)=\delta_{i}(t$) for $i<n$,  and $\delta_{n}(t)$ divides $h_{n}(t)$. Set 
$$\varphi_{1}(t)= \dfrac{h_{n}(t)}{\delta_{n}(t)}.$$
Similarly, we let $\psi_{i}(t)$ denote the Alexander polynomial associated to the torsion $\Gamma$-module $H_{c}^{2n+1-i}(F_{0}; \psi^{S}_{f}\C)$. It was shown in \cite[Theorem 1.1]{Liu} that $\psi_{i}(t)=
\delta_{i}(t$) for $i<n$, and $\delta_{n}(t)$ divides $\psi_{n}(t)$. Set $$\varphi_{2}(t)= \dfrac{\psi_{n}(t)}{\delta_{n}(t)}.$$
As it can be seen from their definitions, the polynomials $\varphi_{1}(t)$ and $\varphi_{2}(t)$ can be regarded as error estimates for the Alexander polynomial $\delta_n(t)$.
The above polynomial invariants are related by the following result:

\begin{thm}\label{t4.1} Assume that the polynomial $f: \CN\rightarrow \C$ is transversal at infinity. With the above notations, we have the following equalities: 
\be\label{one}  r_{n}(t)=\overline{h_{n}(t)}\cdot \psi_{n}(t)= h_{n}(t)\cdot \overline{\psi_{n}(t)},\ee
\be\label{two}  \varphi(t)=\overline{ \varphi_{1}(t)} \cdot \varphi_{2}(t)=\varphi_{1}(t) \cdot \overline{ \varphi_{2}(t)}\ee
\end{thm}

\br\label{sr}
Since the polynomials $h_{i}(t)$, $\delta_i(t)$ and $\psi_{i}(t)$  are products of cyclotomic polynomials (e.g., see \cite{Liu}),  the involution operation $\overline{\ast}$ keeps these polynomials unchanged.
\er

In the course of proving Theorem \ref{t4.1}, we need the following technical result:
\bl\label{tl} We have the following $\Gamma$-module isomorphisms for all $i$:
\be\label{30} H^{2n+1-i}_{c}(H\setminus V\cap H;\R) \cong H_i(F_h),
\ee 
\be\label{31} H^{2n+1-i}(H; \R) \cong H_i(F_h,\partial F_h)
\ee
and 
\be\label{32} H^{2n+1-i}(V\cap H; \R) \cong H_{i-1}(\partial F_h)
\ee
\el

\begin{proof}
Choose coordinates $[x_{0},\cdots ,x_{n+1}]$ for $\CP$, so that $H=\lbrace x_{0}=0 \rbrace$ is the hyperplane at infinity. Then $\mathbb{O}=[0,\cdots,0,1]$ corresponds to the origin in $\CN$. Define
\begin{center}
$\alpha: \CP\rightarrow \mathbb{R}_{+}$, \ $\alpha:= \dfrac{\vert x_{0} \vert^{2}}{\sum_{i=0}^{ n+1} \vert x_{i} \vert ^{2}}.$
\end{center}
Note that $\alpha$  is a well-defined, real analytic and proper function satisfying:
\begin{center}
$0\leq \alpha \leq 1$, $\alpha^{-1}(0)=H$ and $\alpha^{-1}(1)=\mathbb{O}.$
\end{center}
Since $\alpha$ has only finitely many critical values, there exists $\eta$ sufficiently small such that the interval $(0,\eta]$ contains no critical values. Set $$U_{\eta}=\alpha^{-1}[0,\eta).$$ Then $U_{\eta}$ is a tubular neighbourhood of $H$ in $\CP$, and note that $\CP\setminus U_{\eta}$ is a closed large ball of radius $\dfrac{1-\eta}{\eta}$ in $\CN$.  Set $$U_{\eta}^{\ast}=\alpha^{-1}(0,\eta)=U_{\eta}\setminus H.$$
Let us now consider the following commutative diagram of inclusions:
\begin{center}
$\xymatrix{
U_{\eta}^{\ast} \ar[r]^{c}    \ar[d]^{u}          & \CN \ar[d]^{k}  & \U \ar[l]^{l}\ar[dl]^{j}\\
U_{\eta}       \ar[r]^{p}       & \CP &
}$
\end{center}
and restrict the distinguished triangle (\ref{3.5}) over  $U_{\eta}$. We get a triangle:
\be\label{25}
p^*Rk_{\ast} l_{!}\K \rightarrow p^*Rk_{\ast} Rl_{\ast}\K \rightarrow p^*v_{!}(\R\vert_{V}) \overset{[1]}{\rightarrow},
\ee
where $v$ denotes as before the inclusion of $V$ in $\CP$.
Let us first give geometric interpretations to all $\Gamma$-modules appearing in the hypercohomology long exact sequence associated to (\ref{25}).
First note that  
\[
\begin{aligned}
p^* Rk_*  l_!\K[2n+2] 
& \overset{(1)}{\cong} Ru_*c^* l_!\K[2n+2] \\
& \overset{(2)}{\cong} Ru_*c^*k^*k_! l_!\K[2n+2] \\
& \cong Ru_*c^*k^*j_!\K[2n+2] \\
& \overset{(3)}{\cong} Ru_*u^*p^*IC_{\overline{m}}(\CP ,\K) \\
& \cong Ru_*IC_{\overline{m}}(U_{\eta}^{\ast} ,\K)
\end{aligned}
\]
where (1) follows from base change isomorphism  $p^{!}Rk_{\ast}= Ru_{\ast}c^{!} $ (together with  $p^{!}=p^{\ast}$, $c^{!}=c^{\ast}$, as $p$ and $c$ are both open inclusions), for (2) we use the known identity $k^*k_! \cong id$,  and (3) follows from Theorem \ref{thmax}. 

Set $$S_{\infty}=\alpha^{-1}(\eta^{\prime}),$$ for $0<\eta^{\prime}<\eta$.  Then we get as in  \cite[Theorem 4.7]{LM2}:
\[
\begin{aligned}
H^{2n+1-i}(U_{\eta}; p^{\ast}Rk_{\ast}Rl_{!}\K)
& \cong H^{-i-1}(U_{\eta}^{\ast}; IC_{\overline{m}}(U_{\eta}^{\ast} ,\K)) \\
& \overset{(1)}{\cong} H^{-i-1}(S_{\infty};IC_{\overline{m}}(U_{\eta}^{\ast} ,\K)|_{S_{\infty}}) \\
& \cong H_{i}(S_{\infty} \setminus S_{\infty}\cap V; \K) \\
& \overset{(2)}{\cong} H _{i}(F_{h}),
 \end{aligned}
\]
where (1) follows from \cite[Lemma 8.4.7(c)]{KP}, and (2) follows from the fact $V$ intersects $H$ transversally. In fact, the corresponding infinite cyclic cover of  $S_{\infty}\setminus S_{\infty}\cap V$ is homotopy equivalent to $F_{h}$, see \cite[Proposition 4.9]{LM2}. Also note that $IC_{\overline{m}}(U_{\eta}^{\ast} ,\K)|_{S_{\infty}} \cong IC_{\overline{m}}(\CP ,\K)|_{S_{\infty}}$.
Similarly, by using Theorem \ref{thmax} and duality, we have:
\[
\begin{aligned}
H^{2n+1-i}(U_{\eta}; p^{\ast}Rk_{\ast}Rl_{\ast}\K) & \cong H^{-i-1}(S_{\infty};IC_{\overline{l}}(\CP ,\K)|_{S_{\infty}}) \\ &\cong H _{i}(F_{h}, \partial F_{h}).
 \end{aligned}
\]
So, by comparing the hypercohomology long exact sequence associated to (\ref{25}) with the homology long exact sequence induced by the natural inclusion: $\partial F_{h}\rightarrow F_{h}$, and by using the above calculations, we get the following $\Gamma$-module isomorphism:
\be\label{26}
H^{2n+1-i}(S_{\infty};\R|_{S_{\infty}}) \cong H_{i-1}(\partial F_{h}).
\ee

Recall that the triangle (\ref{3.5}) was obtained from (\ref{3.4}) by a base change isomorphism, so the associated hypercohomology long exact sequences for the restrictions of these triangles over $U_{\eta}$ coincide. Note that Lemma 8.4.7(a) of \cite{KP}
 shows that for any $\mathcal{F}^{\cdot}\in D^{b}_{c}(\CP)$, there is an isomorphism:
 \be
H^{\ast}(U_{\eta}; \mathcal{F}^{\cdot} )\cong H^{\ast}(H; \mathcal{F}^{\cdot} )
\ee
So by restricting (\ref{3.4}) over $H$ we get the same hypercohomology long exact sequence as for  restricting (\ref{3.4}) and (\ref{3.5}) over $U_{\eta}$.
Let $i_{hv}$ be the inclusion of $H\cap V$ into $H$.  
By using the proper base change isomorphism (\cite[Theorem 2.3.26]{D2}) for the diagram:
\begin{center}
$\xymatrix{
& H \setminus H\cap V  \ar[r] ^{i_{h}}   \ar[d]         & H \ar[d]  & H\cap V \ar[l]_{i_{hv}} \ar[d]    \\
\U \ar[r]^{k'} & \CP\setminus V         \ar[r]^{l^{\prime}}       & \CP &  V \ar[l]_{v}
}$
\end{center} 
we have (using the notations of Theorem \ref{new}):
\be
(Rl^{\prime}_{!}Rk^{\prime}_{\ast}\K)\vert_{H} = Ri_{h !} ( (Rk^{\prime}_{\ast}\K)\vert_{H\setminus H\cap V}) = Ri_{h !} ( \R\vert _{H \setminus H\cap V} ),\ee
and 
\be
(Rv_{ !} (\R\vert_{V}))\vert _{H} = Ri_{hv !}(\R\vert _{H\cap V}).
\ee
So, the hypercohomology long exact sequence associated to the restriction of the triangle (\ref{3.4}) over $H$ becomes:
$$\cdots \to H^{2n+1-i}_{c}(H\setminus V\cap H;\R) \to H^{2n+1-i}(H; \R) \to H^{2n+1-i}(V\cap H; \R) \to \cdots$$
Therefore, by the above calculations for the restriction of  (\ref{3.5}) over $U_{\eta}$, we get the $\Gamma$-module isomorphisms:
\begin{center}
 $H^{2n+1-i}_{c}(H\setminus V\cap H;\R) \cong H_i(F_h),$
\end{center}  \begin{center}
$H^{2n+1-i}(H; \R) \cong H_i(F_h,\partial F_h)$
\end{center} and \begin{center}
$H^{2n+1-i}(V\cap H; \R) \cong H_{i-1}(\partial F_h)$
\end{center}  for all $i$.
\end{proof}

Let us now get back to the proof of Theorem \ref{t4.1}:

\noindent{\it Proof of Theorem \ref{t4.1}.} \newline
Let us consider the long exact sequence of hypercohomology with compact supports for the peripheral complex $\R$ with respect to the inclusions $$F_{0}\hookrightarrow V\cup H \hookleftarrow H.$$
By using duality and the result in \cite[Theorem 4.7]{LM2} we have that 
$$H_{i}(F_{h},\partial F_{h})\cong \overline{H_{2n-i}(F_{h})} \cong \overline{H_{2n-i}(\U^{c})}$$ for $i> n$.  On the other hand, \cite[Theorem 1.1]{Liu} yields that $$H^{2n+1-i}_{c}(F_{0}; \psi^{S}_{f}\Q) \cong H_i(\U^{c})$$ for $i<n$. By using Proposition \ref{p4.1} and vanishing results for perverse sheaves on affine spaces,  we obtain the $\Gamma$-module isomorphisms:
\begin{equation}\label{4.5}
H^{2n+1-i}_{c}(F_{0}; \psi^{S}_{f}\Q) \cong H^{2n+1-i}(V\cup H; \R), \text{ for } i< n, 
\end{equation}
\begin{equation}\label{4.6}
 H^{2n+1-i}(V\cup H; \R) \cong H^{2n+1-i}(H;\R) \cong H_{i}(F_{h},\partial F_{h}),  \text{ for } i>n,
\end{equation}
and a short exact sequence for $i=n$, \begin{equation}\label{35}
0 \rightarrow H^{n+1}_{c}(F_{0}; \psi^{S}_{f}\C) \rightarrow H^{n+1}(V\cup H; \R) \rightarrow H^{n+1}(H; \R)\rightarrow  0.
\end{equation}
Similarly, for the inclusions $(H \setminus V\cap H)\hookrightarrow V\cup H \hookleftarrow V$, there is a short exact sequence: \begin{equation}\label{36}
0 \rightarrow H^{n+1}_{c}(H \setminus V\cap H; \R) \rightarrow H^{n+1}(V\cup H; \R) \rightarrow H^{n+1}(V; \R) \rightarrow 0.
\end{equation} By using (\ref{31}) and duality, the Alexander polynomial associated to $H^{n+1}(H; \R)$ is $\overline{h_{n}(t)}$. Moreover, since $$\mathcal{D} (\psi^{S}_{f}\C)= (\psi^{S}_{f}\C)^{op}[2n+1],$$
we have by Theorem \ref{new} that 
\be
H^{n+1}(V; \R)\cong H^{n+1}(V; i_{v \ast} \psi^{S}_{f}\Q) \cong H^{n+1}(F_{0}; \psi^{S}_{f}\Q) \cong H^{n+1}_{c}(F_{0}; (\psi^{S}_{f}\Q)^{op}),\ee
where the last isomorphism follows from duality and the universal coefficient theorem. 
So the Alexander polynomial associated to $H^{n+1}(V; \R)$ is $\overline{\psi_{n}(t)}$.

Since $\R$ is supported on $V \cup H$, the result follows now by using the multiplicativity of the Alexander polynomials associated to the short exact sequences (\ref{35}) and (\ref{36}).
\hfill$\square$

\bigskip

We conclude this section with the following degree estimate:
\begin{prop} \label{4.8} Assume that the polynomial $f: \CN\rightarrow \C$ is transversal at infinity. 
\item[(a)] We have the following degree estimates: \begin{center}
$\deg \varphi_{2}\leq  \deg \varphi_{1}\leq d\cdot \mu,$
\end{center} hence $$\deg \varphi\leq 2d\cdot \mu,$$ where $\mu=\vert \chi(\U)\vert$, and $d$ is the degree of $f$.   
\item[(b)] If $F$ denotes the generic fibre of $f$, then $F_{h}$ and $F$ have isomorphic  $\Z$-homology groups.
\item[({c})]
Let $\widetilde{f}$ be the homogenization of $f$, with corresponding Milnor fiber  $\widetilde{F}=\{\widetilde{f}=1\}$. Then, if $\mu =0$, the spaces $\U^{c}$, $F$, $F_{h}$  and $\widetilde{F}$ are all homotopy equivalent to each other. 
\end{prop}
\begin{proof} 
\item[(a)]  Let $\widetilde{F}_{t}$ be the Milnor fibre of $\widetilde{f}$ defined by $\lbrace \widetilde{f}=t\rbrace$, for small enough $t \in \C^{\ast}$. Clearly, $\widetilde{F}_{t}$ is homotopy equivalent to $\widetilde{F}$.
Without loss of generality, $t$ can be chosen so that  $F_{t}=f^{-1}(t)$ is the generic fibre of $f$, hence $F_{t}$ is smooth.  Since $V$ intersects $H$ transversally,  the hyperplane $\lbrace x_{0}=0 \rbrace$ in $\C^{n+2}$ and its parallel hyperplane $\lbrace x_{0}=1 \rbrace$ are both generic for $\widetilde{F}_{t}$ in the sense of \cite{DP}. It follows that, up to homotopy, $\widetilde{F}_{t}$ is obtained from either the Milnor fibre $F_h$ of $h=f_d$, or from the generic fibre $F_{t}$, by attaching $d \cdot \mu$ cells of dimension $n+1$ (\cite[Proposition 9]{DP}). On the other hand, Corollary 6.5 in \cite{Liu} shows that there exists a natural map from $\U^{c}$ to ${\widetilde{F}}$, which induces an $(n+1)$-homotopy equivalence, and in particular, we have that $H_{n}(\U^{c})\cong H_{n}({\widetilde{F}})$. These two facts together yield that  $\deg\varphi_{1} \leq d\cdot \mu$. Note also that by Theorem 1.2 in \cite{Liu} we have  that $ \dim H^{n}_{c}(F_{0},\psi_{f}\Q)\leq \dim H_{n}(F_{t})$. Therefore, $\deg\varphi_{2}\leq \deg \varphi_{1} (\leq d\cdot \mu)$, so $\deg\varphi = \deg\varphi_{1}+\deg\varphi_{2} \leq 2d\cdot \mu$.

\item[(b)] The above homotopy argument yields the following isomorphisms for  $i\leq n-1$:
$$H_{i}(F_{h},\Z)\cong H_{i}(\widetilde{F},\Z)\cong  H_{i}(F,\Z).$$
Since $F_{h}$ and $F$ are $n$-dimensional affine varieties, both of them have the homotopy type of a finite $n$-dimensional CW complex. So $H_{n}(F_{h},\Z)$ and $H_{n}(F,\Z)$ are free Abelian groups, and it remains to show that they have the same rank. This is indeed true since the above discussion shows that $\chi(F_{h})= \chi(F)$.

\item[({c})]   If $\mu=0$, then $F_{h}$ and $F$ are homotopy equivalent to $\widetilde{F}$. In particular, the natural map from $\U^{c}$ to $\widetilde{F}$ induces isomorphisms on all homology groups with $\Z$-coefficients. Since this map is already an $(n+1)$-homotopy equivalence, it follows by Hurewicz's theorem that  $\U^{c}$ is homotopy equivalent to $\widetilde{F}$.
 \end{proof}
\br  In fact, we proved the following two equalities:
 \be
   d \cdot \mu - \deg\varphi_{1} =\dim H_{n+1}({\widetilde{F}}),
   \ee
\be 
\dim H_{n}(F)-\dim H^{n}_{c}(F_{0},\psi_{f}\Q)=\deg\varphi_{1}-\deg\varphi_{2}.
\ee
\er

%%%%%%%%%%%%%%%%%%%%%%%%%%%%%%%%%
%%%%%%%%%%%%%%%%%%%%%%%%%%%%%%%%%

\section{Reidemeister torsion and Alexander polynomials}\label{Rtors}

\subsection{Reidemeister torsion of chain complexes}
In this section, we recall the definition and main results about the Reidemeister torsion, for more details see \cite{KL} and \cite{CF}.

Let $C_{\ast}$ be a finite chain complex: \begin{center}
$\xymatrix{
C_{\ast}=C_{n}\ar[r]^{\partial} & \dots \ar[r]^{\partial} & C_{0}
}$
\end{center} with $C_{i}$ finite dimensional $\F$-vector spaces, and $\partial \circ \partial=0$. Choose a basis $c_{i}$ for $C_{i}$, $\overline{h_{i}}$ a basis for the homology $H_{i}(C_{\ast})$ and $h_{i}$ a lift of $\overline{h_{i}}$ to $C_{i}$. Let $b_{i}$ be a basis for $B_{i}:=$ Image$\xymatrix{
(C_{i+1} \ar[r]^{\partial} & C_{i})
}$ and $\widehat{b}_{i}$ a lift of $b_{i}$ in $C_{i+1}$. If $Z_{i}$ denotes the $i$-cycles, by using the inclusions $B_{i}\subseteq Z_{i} \subseteq C_{i}$ together with the isomorphisms $Z_{i}/ B_{i}\cong H_{i}(C_{\ast})$ and $C_{i}/ Z_{i}\cong B_{i-1}$, it follows that $b_{i} h_{i} \widehat{b}_{i-1}$  is a basis of $C_{i}$. Denote by $[a\vert b]$ the determinant of the transition matrix from the basis $a$ to the basis b.

\begin{definition}(\cite{M3}) The {\it torsion} of $(C_{\ast};c,h)$ is defined as \begin{equation}
\tau(C_{\ast};c,h)=\prod_{ i=0}^{n}  [b_{i}h_{i}\widehat{b}_{i-1}\vert c_{i}]^{(-1)^{i}} \in \F^{\ast} / \lbrace \pm 1\rbrace .
\end{equation} 
\end{definition}

The torsion does not depend on the choice of basis $b$ and its lifts. It depends on the choice of $c$ and $h$ as follows: \begin{equation}
\tau(C_{\ast};c^{\prime},h^{\prime})=\tau(C_{\ast};c,h)\prod_{i} (\dfrac{[h_{i}^{\prime} \vert h_{i}]}{[c_{i}^{\prime} \vert c_{i}]}) ^{(-1)^{i}}.
\end{equation} 

%%%%%%%%%%%%%%%%%%%%%%%%%%%%%%%%%%%%

 \subsection{Torsion and Alexander polynomials}

Let $X$ be a finite connected CW complex, with $\pi=\pi_{1}(X)$. Fix an epimorphism $\rho : \pi \rightarrow \Z$ and note that $\rho$ extends naturally to an epimorphism of algebras $\Z[\pi] \rightarrow \Z[\Z]$ which we also denote by $\rho$. We identity $\Z[\Z]$ with the Laurent polynomial ring $\Z[t,t^{-1}]$. 
If $\widetilde{X} \to X$ is the universal cover of $X$, then the cellular chain complex $C_*(\widetilde{X};\Q)$ is a $\Q[\pi]$-module, freely generated by the lifts of the cells of $X$. Consider the chain complex of the pair $(X,\rho)$ defined as the complex of $\Q[t,t^{-1}]$-modules
$$C_*^{\rho}(X;\Q[t,t^{-1}]):=\Q[t,t^{-1}] \otimes_{\Q[\pi]} C_*(\widetilde{X};\Q).$$
Let $X^{c}$ denote the infinite cyclic cover defined by the kernel of $\rho$. Then under the action of the deck group $\Z$, the chain complex $C_{\ast}(X^{c},\Q)$ becomes a complex of $\Gamma:=\Q[t,t^{-1}]$-modules, which is canonically isomorphic to $C_{\ast}^{\rho}(X; \Q[t,t^{-1}])$.

Denote by $\Q(t)$ the fraction field of $\Q[t,t^{-1}]$ and define 
$$C^{\rho}_{\ast}(X,\Q(t))=C_{\ast}(X^{c},\Q)\otimes_{\Q[t,t^{-1}]}\Q(t).$$
The $i$-th homology of $(X,\rho)$ (also called the {\it $i$-th Alexander module}) is defined to be the $\Q[t,t^{-1}]$-module
$$H^{\rho}_{i}(X,\Q[t,t^{-1}]):=H_i(C_*^{\rho}(X;\Q[t,t^{-1}]))\cong H_{i}(X^{c},\Q),$$
and we extend the definition to $H^{\rho}_{i}(X,\Q(t)):=H_{i}(C^{\rho}_{\ast}(X,\Q(t)))$. Since $\Q[t,t^{-1}]$ is a principal ideal domain and $\Q(t)$ is flat over $\Q[t,t^{-1}]$, it follows that \begin{center}
$H^{\rho}_{i}(X,\Q(t))=H^{\rho}_{i}(X,\Q[t,t^{-1}])\otimes \Q(t)$
\end{center}
Note that the complex $C_{\ast}^{\rho}(X; \Q(t))$ is $\Q(t))$-acyclic if $H^{\rho}_{i}(X,\Q[t,t^{-1}])$ is a torsion $\Gamma$-module for all $i$.

We now define the Reidemeister torsion for the pair $(X,\rho)$. For this, we first note that the complex $C_{\ast}^{\rho}(X; \Q(t))$ is based by construction. 
\begin{definition} Fix a basis for the homology $H^{\rho}_{\ast}(X,\Q(t))$. The {\it  Reidemeister torsion} of $(X,\rho)$ with respect to this basis is defined as 
$$\tau_{\rho}(X)=\tau (C^{\rho}_{\ast}(X,\Q(t))) \in \Q(t)^{\ast}.$$
\end{definition}

We next indicate a basis-free definition for the Reidemeister torsion.
Since $\Q[t,t^{-1}]$ is a PID, any $\Q[t,t^{-1}]$-module $M$ has a decomposition into a direct sum of cyclic modules. Recall that the order of $M$ is defined as the product of the generators of the torsion part.  If the module $M$ is free, the order is 1 by convention. 

\begin{definition} The {\it $i$-th Alexander polynomial} $\delta^{\rho}_{i}(X)$ of $(X,\rho)$ is defined to be the order of $H^{\rho}_{i}(X,\Q[t,t^{-1}])$. 
\end{definition}

The torsion of $(X,\rho)$ can be computed in terms of the Alexander polynomials as follows:
\begin{thm} \label{5.5} (\cite[Theorem 3.4]{KL}) Let $\tau_{\rho}(X)$ be the torsion of $(X,\rho)$ with respect to some basis in homology. Then, up to multiplication by by $ct^{k}$ ($c\in \C^{\ast}$ and $k\in \Z$), we have that \begin{equation}
\tau_{\rho}(X)=\prod _{i} \dfrac{\delta^{\rho}_{2i+1}(X)}{\delta^{\rho}_{2i}(X)}.
\end{equation}
\end{thm}
So we can regard the right hand side of equation (5.3) as a basis-free definition of the Reidemeister torsion.

%%%%%%%%%%%%%%%%%%%%%%%%%%%%%%%%%%%%%%%%

\subsection{Duality and intersection forms}\label{int}

Let $X$ be a smooth compact $2m$-dimensional manifold with boundary $\partial X$. Then it is known that $X$ has a PL structure and any two PL-triangulations have a common linear subdivision which is PL. Endow $X$ with the CW decomposition induced by one of these. Since $X$ is compact, the associated CW complex is finite. Note that $\partial X$ inherits the structure of a PL-manifold from $X$, and the triangulations of $X$ can be used to define the chain complex $C^{\rho}_{\ast}(X,\partial X,\Q(t))$. In this setting one can construct non-singular Poincar\'{e} Duality pairings for $0\leq i \leq 2m$: 
\begin{equation}\label{5.4}
H^{\rho}_{i}(X,\Q(t))\times H^{\rho}_{2m-i}(X,\partial X,\Q(t)) \rightarrow \Q(t)
\end{equation}
The {\it intersection form} of $(X,\rho)$ is the sesquilinear form: \begin{equation}
\phi^{\rho}:H^{\rho}_{m}(X,\Q(t))\times H^{\rho}_{m}(X,\Q(t)) \rightarrow \Q(t)
\end{equation} defined by the composition \begin{center}
$H^{\rho}_{m}(X,\Q(t))\times H^{\rho}_{m}(X,\Q(t)) \rightarrow H^{\rho}_{m}(X,\Q(t))\times H^{\rho}_{m}(X,\partial X,\Q(t)) \rightarrow \Q(t)$
\end{center} where the first map is induced by inclusion and the second map is the pairing (\ref{5.4}).

Let $\overline{\ast}$ denote the canonical involution on $\Q[t,t^{-1}]$. For each $i$, and for a fixed basis on $H^{\rho}_{i}(X,\Q(t))$, we choose the dual basis on $H^{\rho}_{2m-i}(X,\partial X,\Q(t))$ given by (\ref{5.4}). Then we have that (\cite{KL,M2}) \begin{equation}
\tau_{\rho}(X,\partial X)\cdot \overline{\tau_{\rho}(X)}=1
\end{equation}

The following result will be useful later:
\begin{lem} \label{5.6}(\cite{CF}, \cite{KL}) For any $(X^{2m},\rho)$ as above, such that $X$ has the homotopy type of a m-dimensional CW complex and $C_{\ast}^{\rho}(\partial X,\Q(t))$ is acyclic, the following holds: \begin{equation}
\tau_{\rho}(\partial X)=\tau_{\rho}(X)\cdot \overline {\tau_{\rho}(X)}\cdot [\det(\phi^{\rho})]^{(-1)^{m}}.
\end{equation}
\end{lem}

%%%%%%%%%%%%%%%%%%%%%%%%%%%%%%%%%%%%%%%%
\section{Boundary manifold of a hypersurface complement}\label{boundary}
In this section, we define the boundary manifold of a hypersurface complement, and investigate its topology. We make use of the peripheral complex to study the linking number infinite cyclic cover of the boundary manifold. In particular, by using results from Section \ref{per}, we put mixed Hodge structures on the corresponding Alexander modules associated to the boundary manifold.

Choose coordinates $[x_{0},\cdots ,x_{n+1}]$ for $\CP$, and $H=\lbrace x_{0}=0 \rbrace$. Define \begin{center}
$\theta: \CP \rightarrow \mathbb{R}_{+}$, $\theta = \dfrac{\vert \widetilde{f} \vert ^{2}\vert x_{0} \vert^{2}}{(\sum_{i=0}^{n+1} \vert x_{i} \vert ^{2})^{d+1}},$\end{center}
with $\widetilde{f}$ denoting the homogenization of $f$.
Note that $\theta$ is well-defined, it is real analytic and proper, and
$$\theta^{-1}(0)=V\cup H.$$
Since $\theta$ has only finitely many critical values, there exists a positive real number $\varepsilon$ sufficiently small such that the interval $(0,\varepsilon]$ contains no critical values. Set $$\U_{0}=\theta^{-1}\left([\varepsilon,+\infty)\right).$$ So $\U_0$ is the complement of the regular tubular neighbourhood of $V\cup H$ in $\CP$.

Moreover, $\U_{0}$ is a manifold with boundary, homotopy equivalent to $\U$ (e.g., see \cite[page 149]{D1}). 
Note that while $\U_{0}$ has the homotopy type of a finite $(n+1)$-dimensional CW complex, its boundary $\partial \U_{0}$ is a smooth compact ($2n+1$)-real dimensional manifold. The inclusion $\partial \U_{0} \hookrightarrow \U_{0}$ is an $n$-homotopy equivalence (\cite[(5.2.31)]{D1}), hence  $\pi_{i}(\partial \U_{0})=\pi_{i}(\U_{0})$ for $i<n$, and we have an epimorphism: $\pi_{n}(\partial \U_{0})\twoheadrightarrow \pi_{n}(\U_{0})$. Therefore 
$$\rho: \pi_{1}(\partial \U_{0})\twoheadrightarrow \pi_{1}(\U_{0})= \pi_{1}(\U)\overset{f_*}{\twoheadrightarrow}  \pi_{1}(\C^{\ast})=\Z$$
 is an epimorphism, which defines the infinite cyclic cover $(\partial \U_{0})^{c}$ of $\partial \U_{0}$.
 
 We refer to $\partial \U_{0}$ as the {\it boundary manifold} of the hypersurface complement $\U$. For a study of topological properties of such boundary manifolds, see \cite{CoS1,CoS2}.

\begin{prop}  \label{6.1} We have the following $\Gamma$-module isomorphisms: 
\be\label{6.1} H_{i}^{\rho}(\partial \U_{0},\Q[t,t^{-1}])\cong H_{i}((\partial \U_{0})^{c})\cong   H^{2n+1-i}(V\cup H;\R).\ee  
In particular, $H_{i}((\partial \U_{0})^{c})$ is a torsion $\Gamma$-module and $C_{\ast}^{\rho}(\partial \U_{0},\Q(t))$ is $\Q(t)$-acyclic. Moreover, the zeros of the Alexander polynomial associated to $H_{i}((\partial \U_{0})^{c})$ are roots of unity for all  $i$, and have order $d$ except for $i=n$. Finally, $H_{i}((\partial \U_{0})^{c})$ is a semi-simple $\Gamma$-module for $i\neq n$.
\end{prop}
\begin{proof}
We have the following isomorphisms of $\Gamma$-modules (see \cite[Corollary 3.4, Lemma 3.5]{LM2}) \be\label{3}
 IH_{i}^{\overline{m}}(\CP,\K)\cong H_{i}(\U,\K)\cong H_{i}(\U_{0},\K)\cong H_{i}^{\rho}(\U_{0},\Q[t,t^{-1}]),\ee
 and 
 \be\label{4}
IH_{i}^{\overline{l}}(\CP; \K)\cong H^{BM}_{i}(\U;\K)\cong H^{\rho}_{i}(\U_{0}, \partial \U_0;\Q[t,t^{-1}]),\ee
 where $H^{BM}_{\ast}$ denotes the Borel-Moore homology, and the last isomorphism in (\ref{4}) follows by Poincar\'e duality and homotopy equivalence. So, by comparing the homology exact sequence (with $\Q[t,t^{-1}]$-coefficients) of the pair $(\U_{0},\partial \U_{0})$ with the hypercohomology long exact sequence of the distinguished triangle defining the peripheral complex,  we obtain the following isomorphism of $\Gamma$-modules:
 \be\label{5} H_{i}^{\rho}(\partial \U_{0},\Q[t,t^{-1}]) \cong H^{2n+1-i}(\CP;\R).\ee 
 And since $\R$ is supported on $V \cup H$, we obtain the isomorphisms in (\ref{6.1}).
  Moreover, as $\R$ is a $\Gamma$-torsion sheaf complex, it follows from (\ref{5}) that $H_{i}((\partial \U_{0})^{c}) \cong H_{i}^{\rho}(\partial \U_{0},\Q[t,t^{-1}])$ is a $\Gamma$-torsion module. In particular, the chain complex $C_{\ast}^{\rho}(\partial \U_{0},\Q(t))$ is $\Q(t)$-acyclic. The zeros of the corresponding Alexander polynomials are roots of unity since this is the case for the Alexander polynomials associated to the modules $H^{2n+1-i}(V\cup H;\R)$ (e.g., see \cite{LM2}).
The remaining claims follow from Theorem \ref{t2.2} and Proposition \ref{p4.1}.
 \end{proof}

The following result is a direct consequence of Proposition \ref{6.1} and Corollary \ref{3.3}.  
 \begin{cor}\label{6.2} The Alexander modules $H_{i}((\partial\U_{0})^{c} )$ of the boundary manifold $\partial\U_{0}$ are endowed with mixed Hodge structures induced from the peripheral complex, for all $i$. 
Moreover, for $i\neq n$, this mixed Hodge structure is compatible with the $\Gamma$- action, i.e.,  $t: H_{i}((\partial\U_{0})^{c})\rightarrow H_{i}((\partial\U_{0})^{c})$ is a mixed Hodge structure  homomorphism for $i\neq n$.
 \end{cor}
\begin{proof}  By Corollary \ref{3.3}, the peripheral complex $\R$ underlies a (shifted) mixed Hodge module. Hence the $\Q$-vector space isomorphism (underlying the  $\Gamma$-module isomorphism  of Proposition \ref{6.1}) \be
H_{i}((\partial \U_{0})^{c})\cong H^{2n+1-i}(V\cup H;\R)\ee  defines a mixed Hodge structure on $H_{i}((\partial \U_{0})^{c})$, for all $i$.

In order to prove the second claim, note that by 
 \cite{Liu}, the mixed Hodge structure on \begin{center}
 $H^{2n-i}_{c}(F_{0};\psi_{f}\C) \cong H_{i}(F_{h})$  for $i< n$
\end{center}    is compatible with the $\Gamma$-action. Then the isomorphism  (\ref{4.5}) shows that the resulting mixed Hodge structure on $H_{i}((\partial\U_{0})^{c} )$ has the same property for  $i< n$. 
 
By Alexander duality on $H_{\ast}(F_{h})$, the mixed Hodge structure on $H_{i}(F_{h},\partial F_{h})$ is compatible with the $\Gamma$-action for $i > n$. Then, by using the isomorphism (\ref{4.6}),  the resulting mixed Hodge structure on $H_{i}((\partial\U_{0})^{c} )$ satisfies the same property for $i>n$.
\end{proof} 

Our next result gives a geometric interpretation of the homology of the boundary manifold.

Let $g=x_{0}\widetilde{f}$ be the homogeneous polynomial of degree $d+1$ whose zero locus is the divisor $V\cup H$. Consider the associated Milnor fibre $F_{g}= \lbrace g=1 \rbrace$ and its boundary manifold $\partial F_{g}$. Then there exists a natural $(d+1)$-fold covering map (cf. \cite[page 149]{D1}):  \be\label{d1} \partial F_{g} \to \partial \U_{0}.\ee
\begin{prop} \label{6.9} The covering map (\ref{d1}) induces isomorphisms of $\Q$-vector spaces:
\be \label{17}
 H_{i}(\partial F_{g})\cong H_{i}(\partial \U_{0}) \text{ for    } i\neq n, n+1.
\ee 
 Moreover, if the complex numbers $\lambda_{\alpha}=\exp(\dfrac{2\pi i \alpha}{d+1})$, with $\alpha =1,2,\cdots, d$,  are not among the roots of $\psi_{n}(t)$, then the isomorphism (\ref{17}) holds for all $i$.
In particular, this is the case if  $\mu=0$ (e.g., $f$ is homogeneous).
\end{prop}
\begin{proof} Denote by $N(\lambda,i)$ the number of direct summands in the $(t-\lambda)$-torsion part of $H_{i}((\partial \U_{0})^{c}; \C)$, i.e., the number of the Jordan blocks with eigenvalue $\lambda$ for the automorphism on $H_{i}((\partial \U_{0})^{c}; \C)$ induced by the $\Gamma$-action. Define a rank one local system $\K_{\lambda}$ on $\partial\U_{0}$ by the composed map: $\pi_{1}(\partial\U_{0})\overset{\rho}{\to} \Z\rightarrow \C^{\ast}$, where the last map is defined by $1_{\Z}\mapsto \lambda$.   If $\lambda=1$, then $\K_{\lambda}=\C$.
The Wang exact sequence 
$$\cdots \to H_{i}((\partial \U_{0})^{c}; \C) \overset{t - \lambda}{\to} H_{i}((\partial \U_{0})^{c}; \C) \to H_{i}(\partial \U_{0}; \K_{\lambda}) \to H_{i-1}((\partial \U_{0})^{c}; \C) \to \cdots $$ yields that (e.g., see \cite[Theorem 4.2]{DN})
 \be \label{43}
\dim H_{i}(\partial \U_{0}; \K_{\lambda})=N(\lambda,i)+N(\lambda,i-1),
\ee   
for all $i$. On the other hand, by using the $(d+1)$-fold covering map (\ref{d1}), 
we have  that  
\be H_{i}(\partial F_{g};\C)\cong \oplus_{\lambda^{d+1}=1} H_{i}(\partial \U_{0}; \K_{\lambda}) .\ee
If $\lambda^{d}\neq 1$, then it follows from Proposition \ref{6.1} that $N(\lambda,i)=0$ for $i\neq n$. (Note that $\gcd (d, d+1)=1$.)  So by using (\ref{43}), we get the isomorphism (\ref{17}) for $ i\neq n, n+1$.

Moreover, if the complex numbers $\lambda_{\alpha}=\exp(\dfrac{2\pi i \alpha}{d+1})$, with $\alpha =1,2,\cdots, d$,  are not among the roots of $\psi_{n}(t)$, then  the short exact sequence (\ref{35}) shows that  $N(\lambda,n)=0$ for $\lambda\in \{\lambda_{\alpha} | \ \alpha =1,2,\cdots, d\}$. So, in view of  (\ref{43}),  the isomorphism (\ref{17}) holds also for $i=n, n+1$.
In particular, this is the case when $\mu=0$, since by Proposition 4.2 in \cite{Liu} it follows that $\psi_{n}(t)=h_{n}(t)$ has only roots of unity of order $d$.\end{proof}

\br  The natural inclusion $\partial \U_{0} \hookrightarrow \U_{0}$ is an $n$-homotopy equivalence, and so is the inclusion $\partial F_{g} \hookrightarrow F_{g}$ (see \cite[(3.2.4)]{D2}). Then Proposition \ref{6.9} yields that
$H_{i}(F_{g})\cong H_{i}( \U)$ for $i< n$ (compare with \cite[Theorem 1.4]{DL}). 
\er
\bex  Let $V\cup H$ be the hypersurface in $\CP$ defined by $g=x_{0}x_{1}\cdots x_{n+1}$.  Then both $\partial F_{g}$ and $\partial \U_{0}$ are homotopy equivalent to $S^{n}\times (S^{1})^{n+1}$ (see \cite[(5.2.29)]{D1}).
\eex

%%%%%%%%%%%%%%%%%%%%%%%%%%%%%%%%%%%%%%%%

\section{Alexander Polynomial estimates via Reidemeister torsion}\label{Alexid}
In this section, we refine the error estimates for Alexander polynomials given in Section \ref{error} by making use of Reidemeister torsion and the intersection form.

Proposition \ref{6.1} of the previous section can be used to prove the following refinement of Theorem \ref{t4.1}:

\begin{thm}\label{t6.2}  Assume that the degree $d$ polynomial $f: \CN\rightarrow \C$ is transversal at infinity. Let $\phi^{\rho}$ be the intersection form for $(\U_{0},\partial \U_{0})$ associated to $\rho$. Then, with the notations from Section \ref{error}, we have:
\be\label{6}
\det(\phi^{\rho})=\varphi(t),\ee
 and
\be\label{7}
h_{n}(t)\cdot \psi_{n}(t)= \delta_{n}^{2}(t)\cdot \det(\phi^{\rho}).\ee
Moreover, $\deg(\det(\phi^{\rho}))=\deg\varphi(t) \leq 2 d \cdot \mu$, where $\mu =\vert\chi(\U) \vert$.
\end{thm} 

\begin{proof}
By Theorem \ref{5.5}, we have:
\be\label{8}\tau_{\rho}(\U_{0})=\tau_{\rho}(\U)=\prod _{i=0}^{n} \delta_{i}(t)^{(-1)^{i+1}}.
\ee
Since $\U_{0}$ has the homotopy type of a finite (n+1)-dimensional CW complex, and the complex $C_{\ast}^{\rho}(\partial \U_{0},\Q(t))$ is $\Q(t)$-acyclic, Lemma \ref{5.6}  yields the following Alexander polynomial identity:
\be\label{9}
\tau_{\rho}(\partial \U_{0} )=\prod _{i=0}^{n} \lbrace\delta_{i}(t)\cdot \overline{\delta_{i}(t)}\rbrace^{(-1)^{i+1}}\cdot [\det(\phi^{\rho})]^{(-1)^{n+1}}.\ee
On the other hand, by using Theorem \ref{5.5} and Proposition \ref{6.1}, we have: \be\label{10} \tau_{\rho}(\partial \U_{0} )= \prod _{i=0}^{2n} r_{i}(t)^{(-1)^{i+1}}.\ee
Recall that, by equation (\ref{4.3}), the polynomials $r_i(t)$ and $\delta_i(t)$ are related by:
 \begin{equation}\label{11}
 r_{i}(t)=\left\{ \begin{array}{ll}
\delta_{i}(t), & i<n, \\
\overline{ \delta_{2n-i}(t)}, & i>n.\\
\end{array}\right.
 \end{equation}
So, by plugging (\ref{10}) and (\ref{11}) in formula (\ref{9}), we obtain:
\be\label{12} r_{n}(t)= \delta_{n}(t) \cdot \overline{\delta_{n}(t)}\cdot \det(\phi^{\rho}).\ee
Therefore, \be \det(\phi^{\rho})=\frac{r_n(t)}{ \delta_{n}(t) \cdot \overline{\delta_{n}(t)}}=\varphi(t),\ee
and, by using Theorem \ref{t4.1} and Remark \ref{sr}, we get the identity:
\be h_{n}(t)\cdot \psi_{n}(t)= \delta_{n}^{2}(t)\cdot \det(\phi^{\rho}).\ee
The degree estimate $\deg(\det(\phi^{\rho}))=\deg\varphi \leq 2 d\cdot \mu$ follows from Proposition \ref{4.8}.
\end{proof}

\begin{remark}
Since $V$ intersects $H$ transversally, we can take the (closed) regular neighborhood $N$ of $V\cup H$ to be the union of a regular neighborhood of $V$, say $N(V)$, with a tubular neighborhood of the hyperplane at infinity (after rounding corners). Then \begin{equation}\partial \U_{0} = \left(S^{2n+1}_{R} \setminus (S^{2n+1}_{R}\cap N^{\circ}(V))\right) \cup \left(B_{R} \cap \partial N(V)\right),\end{equation} where $B_{R}$ is a closed large ball of radius $R$ in $\CN$ with boundary sphere $S^{2n+1}_{R}$.
In \cite[Proposition 4.9]{LM2} it is shown that the infinite cyclic cover of $S^{2n+1}_{R} \setminus (S^{2n+1}_{R}\cap N(V))$ is homotopy equivalent to the Milnor fiber $F_{h}$. 
Moreover, if 
 $\left(B_{R} \cap \partial N(V)\right)^{c}$ denotes the corresponding infinite cyclic cover of $B_{R} \cap \partial N(V)$, it follows as in Lemma \ref{tl} that  \begin{equation}
 H_{i}(\left(B_{R} \cap \partial N(V)\right)^{c})\cong H_{c}^{2n+1-i}(F_{0}, \psi_{f}^{S}\Q), \end{equation} for all $i$.
These two facts together give a geometric interpretation of Theorem \ref{t6.2}, which is also consistent with the proof of \cite[Theorem 5.6]{CF}.
\end{remark}

\begin{remark}
 If $\mu=0$, then the chain complex $C_{\ast}^{\rho}(\U_{0},\Q(t))$ is $\Q(t)$-acyclic, thus the intersection pairing of Section \ref{int} is trivial. In this case, we have:  $\det(\phi^{\rho})=1$. This fact, coupled with the previous theorem,  gives another proof of the result obtained in  (\cite{Liu}), asserting that: \be
\mu=0 \Rightarrow \varphi=\varphi_{1}=\varphi_{2}=1.\ee
\end{remark}

\begin{example} If $\mu=0$ (e.g., $ f$ is homogeneous), then  $\delta_{i}(t)=h_{i}(t)$ for all $i$. Then, it is known that (cf. \cite[ (4.1.21)]{D1}) 
\be \prod^{n}_{i=0}h_{i}(t)^{(-1)^{i+1}}=(t^{d}-1)^{-\chi(F_{h})/d}.\ee
So,
\be \tau_{\rho}(\partial \U_{0})=(t^{d}-1)^{-2\chi(F_{h})/d}.\ee
\end{example}

As an application of Theorem \ref{t6.2}, we have the following:
\begin{cor}\label{cor1} Assume that the polynomial $f: \CN\rightarrow \C$ is transversal at infinity, and the hypersurface $F_{0}$ has only isolated singularities. Then we have the following equality:
\begin{equation}\label{13}
(t-1)^{\mu+(-1)^{n+1}}(t^{d}-1)^{\xi}\cdot \prod_{p\in Sing(F_{0})}\Delta_{p}(t)= \delta_{n}(t)^{2} \cdot det(\phi^{\rho}),
\end{equation}
where $ \Delta_{p}(t)$ is the (top) local Alexander polynomial associated to the singular point $p\in Sing(F_{0})$, and $\xi=\dfrac{(d-1)^{n+1}+(-1)^{n}}{d}$.  In particular, $\deg(det(\phi^{\rho}))=\deg\varphi$ is even.
\end{cor}
\begin{proof}  We only need to compute the polynomials $\psi_{n}(t)$ and $h_{n}(t)$. For the case of isolated singularities, Section 5.2.1 in \cite{Liu} gives the following equality:
\be\label{14}
\psi_{n}(t)= (t-1)^{\mu}\prod_{p\in Sing(F_{0})}\Delta_{p}(t),\ee
while \cite[(4.1.23)]{D1} provides another equality:
\be\label{15} h_{n}(t)= (t-1)^{(-1)^{n+1}}(t^{d}-1)^{\xi},\ee
 where $\xi=\dfrac{(d-1)^{n+1}+(-1)^{n}}{d}$. Then (\ref{13}) follows from Theorem \ref{t6.2}.

Since $F_0$ has only isolated singularities, $V\cap H$ is a smooth hypersurface in $H$. Then $\lbrace h+x_{0}^{d}=0\rbrace$ is a smooth degree $d$ hypersurface in $\CP$.  Corollary (5.4.4) in \cite{D1} shows that $\chi(F_{0},\psi_{f}\Q)+\chi(H\cap V)$ equals the Euler characteristic number of any smooth degree $d$ hypersurface in $\CP$, so, in particular,
\be \chi(F_{0},\psi_{f})+\chi(H\cap V)=\chi(\lbrace h+x_{0}^{d}=0\rbrace).\ee
Note  also that
\be \chi(F_{h})+\chi(H\cap V)=\chi(\lbrace h+x_{0}^{d}=0\rbrace).\ee
By the last two identities, we get  $\chi(F_{h})=\chi(F_{0},\psi_{f}\Q)$, which shows that $\deg\varphi_{1}=\deg\varphi_{2}$. So $\deg\varphi=\deg\varphi_{1}+\deg\varphi_{2}\equiv 0$ (mod 2).
\end{proof}

\begin{remark} \label{67}
$(a)$ When $n=1$, Corollary \ref{cor1} is also a consequence of \cite[Corollary 5.8]{CF}. \\
$(b)$ It would be interesting to see if the above property of (the degree of determinant of) the intersection pairing remains valid if $f$ has arbitrary singularities. \\
$(c)$ For the case of isolated singularities, $\mu$ is given by the formula: \be \label{20}
\mu=(d-1)^{n+1}-\sum_{p\in Sing(F_{0})} \mu_{p},
\ee where $\mu_{p}$ is the Milnor number of $f$ at $p$ (see \cite{DP}). Then the degree estimates of Theorem \ref{t6.2}, together with Corollary \ref{cor1},  yield that:
$$ 2(d-1)^{n+1}=2\deg \delta_{n}(t)+\deg (\det(\phi^{\rho})) \leq 2\deg \delta_{n}(t)+2d\cdot \mu.$$
Therefore, \be \deg \delta_{n}(t) \geq (d-1)^{n+1}-d\cdot \mu.\ee In particular, we obtain a non-vanishing result for $H_{n}(\U^{c})$ for small $\mu$. Such examples (with $\mu=1, 2$) are given in \cite{Huh}.
\end{remark}

\bex If $F_{0}$ is smooth, then $\delta_{n}(t)=1$ (see \cite[Lemma 1.5]{L3}) and $\mu=(d-1)^{n+1}$. So, by Corollary \ref{cor1}, we conclude that $$\det(\phi^{\rho})=(t-1)^{(d-1)^{n+1} +(-1)^{n+1}}(t^{d}-1)^{\xi},$$ where $\xi=\dfrac{(d-1)^{n+1}+(-1)^{n}}{d}.$
\eex

\bex In relation to Remark \ref{67}({c}), consider the hypersurface in $\mathbb{CP} ^{3}$ defined by  $V= \lbrace x_{0}x_{1}x_{2}+x_{3}^{3}=0\rbrace $.  Then $V$ has only isolated singularities, and it is known that $\mu=2$ for any generic hyperplane at infinity (see \cite[Conjecture 20]{Huh}).  So, $\deg \delta_{2}(t) \geq (3-1)^{3}-3 \cdot 2=2$. 
\eex

\bex Consider the hypersurface in $\mathbb{CP} ^{2}$ defined by  $V= \lbrace x_{0}x_{1}^{3}+x_{1}^{4}+x_{2}^{4}=0\rbrace $. The singular locus of $V$ is just one point $p=[1, 0, 0]$. Let $H=\lbrace x_{0}=0 \rbrace$ be the hyperplane at infinity.  Note that  $V \cap H$ is a smooth hypersurface in $H$, hence $V$ is transversal to $H$. The link pair of the point $p$ in $(\mathbb{CP} ^{2}, V )$ is obtained by intersecting the affine variety $x_{1}^{3}+x_{1}^{4}+x_{2}^{4} = 0$ in $\C^{2}$ with a small sphere about the origin. Since we work in a neighborhood of the origin, by an analytic change of coordinates, this is same as the link pair of the origin in the variety $x_{1}^{3}+x_{2}^{4} = 0$. The polynomial $x_{1}^{3}+x_{2}^{4}$ is weighted homogeneous with weighted degree $12$ for the weight $(4,3)$, and the characteristic polynomial of the monodromy homeomorphism of the associated Milnor fibration is $(t^{4}-t^{2}+1)(t^{2}-t+1)$. So the link of the singular point is a rational homology sphere, which in turn yields that $F_{0}$ is a rational homology manifold. In particular, $\delta_{1}(1)\neq 0$ (e.g., see\cite[Corollary 5.4]{Liu}). Also note that the local Alexander polynomial of the link of the singularity has prime divisors none of which divides $t^{4}-1$. Thus, they cannot be among the prime divisors of $\delta_{1}(t)$ (see \cite{L3}), hence $\delta_{1}(t)=1$.  Formula (\ref{20}) yields that $\mu=3$. By Corollary \ref{cor1}, we conclude that $$\det(\phi^{\rho})=(t-1)^{4}(t^{4}-1)^{2}(t^{4}-t^{2}+1)(t^{2}-t+1).$$
Note also that $\gcd(12,5)=1$, so it follows from by Proposition \ref{6.9} that $\partial F_{g}$ and $\partial \U_{0}$ have same rational homology groups, where $g=x_{0}(x_{0}x_{1}^{3}+x_{1}^{4}+x_{2}^{4})$.
\eex

%%% ====================== End of main part ====================== %%%

%------------------------------------------------------------------
\end{document}